\documentclass[11pt]{amsart}
\usepackage{graphicx,amssymb,epstopdf,subcaption,mathtools,tikz-cd,pb-diagram,thmtools, thm-restate,chngcntr,pinlabel,pifont,enumitem}
\calclayout

\newcommand{\nc}{\newcommand}
\nc{\dmo}{\DeclareMathOperator}
\nc{\nt}{\newtheorem}

\nt{theorem}{Theorem}
\nt{corollary}[theorem]{Corollary}
\nt{example}[theorem]{Example}
\nt{lemma}[theorem]{Lemma}
\nt{criterion}[theorem]{Criterion}
\nt*{theorem*}{Theorem}
\nt{question}[theorem]{Question}
\nt{problem}[question]{Problem}
\nt{conjecture}[question]{Conjecture}
\nt{proposition}[theorem]{Proposition}

\dmo{\Homeo}{Homeo}
\dmo{\Mod}{Mod}
\dmo{\SMod}{SMod}
\dmo{\PMod}{PMod}
\dmo{\I}{\mathcal{I}}
\dmo{\E}{Ends}

\nc{\Z}{\mathbb Z}
\nc{\R}{\mathbb R}
\nc{\N}{\mathbb N}

\nc{\C}{\mathcal C}
\nc{\G}{\mathcal G}
\nc{\f}{\mathfrak f}

\nc{\GE}{\text{Ends}_g}
\nc{\g}{\text{g}}

\nc{\cv}{\mathcal{V}}

\nc{\margin}[1]{\marginpar{\scriptsize #1}}

\title[Centers of subgroups of big mapping class groups
and the Tits alternative]{Centers of subgroups of big mapping class groups\\
and the Tits alternative}
\author{Justin Lanier}
\author{Marissa Loving}

\address{Justin Lanier \\ School of Mathematics\\ Georgia Institute of Technology \\ 686 Cherry St. \\ \mbox{Atlanta}, GA 30332 \\  jlanier8@math.gatech.edu}

\address{Marissa Loving \\ Department of Mathematics\\ University of Illinois\\ 1409 W. Green Street \\ \mbox{Urbana}, IL 61801 \\  mloving2@illinois.edu}

\thanks{The authors are supported by NSF Grants DGE - 1650044 and DGE - 1144245.}

\begin{document}

\vspace*{-1in}
\maketitle

\vspace*{-.4in}

\begin{abstract}
In this note we show that many subgroups of mapping class groups of infinite-type surfaces without boundary have trivial centers, including all normal subgroups. Using similar techniques, we show that every nontrivial normal subgroup of a big mapping class group contains a nonabelian free group. In contrast, we show that no big mapping class group satisfies the strong Tits alternative enjoyed by finite-type mapping class groups. We also give examples of big mapping class groups that fail to satisfy even the classical Tits alternative and give a proof that every countable group appears as a subgroup of some big mapping class group.
\end{abstract}

%%%
%%%
%%%

\vspace*{.4in}

Let $S$ be a connected orientable surface without boundary that is of infinite-type, so that $\pi_1(S)$ is infinitely generated. The mapping class group $\Mod(S)$ is the group of homotopy classes of orientation-preserving homeomorphisms of $S$. The mapping class group of an infinite-type surface is often called a big mapping class group. Similarly, let $S_g$ be the compact connected orientable surface of genus $g$ and $\Mod(S_g)$ be its mapping class group.

In this note we address several questions about subgroups of $\Mod(S)$. We first prove some results about the triviality of centers of subgroups of $\Mod(S)$. We then show that $\Mod(S)$ never satisfies the strong Tits alternative enjoyed by $\Mod(S_g)$, as well as some related results.

\subsection*{Centers}The center of $\Mod(S_g)$ is trivial for $g \geq 3$, while its center is a cyclic group generated by the hyperelliptic involution when $g= 1$ or $2$ \cite[Section 3.4]{Primer}. The centers of mapping class groups of finite-type surfaces with punctures and boundary components were computed by Paris--Rolfsen \cite[Theorem 5.6]{PRmcg}. Although Dehn twists about boundary components are always central, for finite-type surfaces without boundary there are only finitely many exceptional cases of low complexity where the center is nontrivial.

These results where the center is trivial follow from the existence of a (stable) Alexander system---a collection $\Gamma$ of simple closed curves and arcs in the surface such that no nontrivial mapping class fixes $\Gamma$ \cite[Section 2.3]{Primer}. Hern\'andez--Morales--Valdez recently proved that the Alexander method carries over to the infinite-type setting \cite{HMV}.

\begin{theorem}[Hern\'andez--Morales--Valdez] \label{theorem:alexander}
Let $S$ be an orientable surface of infinite topological type, with possibly non-empty boundary. There exists a collection of essential arcs and simple closed curves $\Gamma$ on $S$ such that any orientation-preserving homeomorphism fixing pointwise the boundary of $S$ that preserves the isotopy classes of the elements of $\Gamma$, is isotopic to the identity.
\end{theorem}

With this result in hand, it is straightforward to compute the centers of big mapping class groups, just as in the finite-type case. Throughout this note, we restrict our attention to surfaces without boundary. When $S$ has boundary components, Dehn twists about boundary components are always central, and the center of $\Mod(S)$ can be analyzed by applying a capping homomorphism to each boundary component \cite[Proposition~3.19]{Primer}. 

\begin{proposition} \label{proposition:center}
If $S$ is an infinite-type surface without boundary, the center of $\Mod(S)$ is trivial.
\end{proposition}
\begin{proof}
Suppose $f \in \Mod(S)$ is nontrivial; we show that $f$ is not central. As a consequence of Theorem \ref{theorem:alexander}, there exists a curve $c$ such that $f(c) \neq c$. Consider $fT_cf^{-1}$, which equals $T_{f(c)}$. If $f$ were central, the product $fT_cf^{-1}$ would also equal $T_c$. But $T_{f(c)} \neq T_c$, since curves are compact and $f(c) \neq c$.\end{proof}

If we now consider subgroups of $\Mod(S)$, a similar proof goes through for any subgroup that contains every Dehn twist, or even some nonzero power of every Dehn twist. Examples of such subgroups include: the pure mapping class group $\PMod(S)$, which is the subgroup that acts trivially on the space of ends of $S$; the compactly supported mapping class group $\Mod_c(S)$; and the level $m$ subgroup $\Mod(S)[m]$, which is the subgroup that acts trivially on $\text{H}_1(S; \Z/m\Z)$. Thus all of these subgroups of $\Mod(S)$ have trivial center.

Patel--Vlamis showed that the center of $\PMod(S)$ is trivial in the case where $S$ is an infinite-type surface with finite genus and without boundary \cite[Lemma 3.3]{PV}. Their proof uses the same standard trick of conjugating Dehn twists, but leverages the compact-open topology on $\Mod(S)$ as a stand-in for the Alexander method. 

Note that $\PMod(S)$, $\Mod_c(S)$, and $\Mod(S)[m]$ are all normal in $\Mod(S)$. In fact, we can extend the result of Proposition \ref{proposition:center} to the centers of all normal subgroups of $\Mod(S)$.

\begin{theorem}
If $S$ is an infinite-type surface without boundary, every normal subgroup of $\Mod(S)$ has trivial center.
\label{thm:normaltrivialcenter}
\end{theorem}

\begin{proof}
Let $N$ be any normal subgroup of $\Mod(S)$ and let $f$ be any nontrivial element of $N$. As $f$ is nontrivial, by Theorem~\ref{theorem:alexander} there exists a curve $c$ such that $f(c) \neq c$. Since $N$ is normal, the product $f(T_c f^{-1} T_c^{-1})=T_{f(c)}T^{-1}_c$ is also in $N$. We will show that $f$ and $T_{f(c)}T^{-1}_c$ do not commute.
Conjugating $T_{f(c)}T^{-1}_c$ by $f$ yields $T_{f^2(c)}T^{-1}_{f(c)} 
\in N$. If $T_{f(c)}T^{-1}_c$ and $f$ were to commute, we would have $T_{f^2(c)}T^{-1}_{f(c)}=T_{f(c)}T^{-1}_c$. But this is not possible. 
First, consider the case when $i(c,f(c))=0$. Since $c$ and $f(c)$ are distinct, so are $f(c)$ and $f^2(c)$. But then the two products are clearly distinct multitwists, even if $f$ swaps $c$ and $f(c)$.
Then consider the case when $i(c,f(c)) \geq 1$. Rearranging the supposed equality yields $T_{f^2(c)}=T_{f(c)}T^{-1}_cT_{f(c)}$. The product on the right-hand side is a partial pseudo-Anosov on the surface filled by $c$ and $f(c)$, since its conjugate $T_{f(c)}^{2}T^{-1}_c$ is pseudo-Anosov by Penner's construction \cite{Penner}; see also \cite[p. 396]{Primer}. However, the left-hand side of the equation is a Dehn twist, a contradiction. We therefore have that $f$ does not commute with $T_{f(c)}T^{-1}_c$ and so conclude that the center of $N$ is trivial.
\end{proof}

We note that this argument also holds for finite-type surfaces with stable Alexander systems and recovers the corresponding fact about normal subgroups in that setting. This is a well-known fact for finite-type mapping class groups. A proof can be given by considering the action of the normal subgroup on the space of projective measured foliations; see, for instance, \cite[p. 52]{BM}. Theorem \ref{thm:normaltrivialcenter} can be also be proved by appealing to the finite-type result, in a similar way to how we proceed in our proof of Proposition~\ref{prop:freegroup}.

\subsection*{The Tits alternative}Of course, for any $S$ there do exist subgroups of $\Mod(S)$ that have nontrivial center. For instance, $\Mod(S)$ has abelian subgroups, such as any nontrivial cyclic subgroup or more generally any subgroup generated by mapping classes supported on disjoint subsurfaces. It is a theorem, proved independently by Ivanov and McCarthy, that $\Mod(S_g)$ satisfies a strong version of the Tits alternative: every subgroup is either virtually abelian or contains a nonabelian free group \cite{Ivanov, McCarthy}. In contrast, big mapping class groups do not satisfy this strong version of the Tits alternative.

\begin{theorem}
\label{theorem:ZwrZ}
For every surface $S$ of infinite type without boundary, $\Mod(S)$ contains the restricted wreath product $\Z \wr \Z$ as a subgroup and so does not satisfy the strong Tits alternative.
\end{theorem}
\begin{proof} By the classification theorem of Ker\'{e}kj\'{a}rt\'{o} and Richards \cite{Ker, Richards}, every surface $S$ of infinite type without boundary contains at least one of the following: infinite genus, a countable collection of isolated punctures, or a Cantor set of punctures where none of these punctures is accumulated by isolated punctures. We consider these possibilities in turn.

Let $S$ be a surface with infinite genus. Then $S$ has at least one end accumulated by genus. Consider the subgroup of $\Mod(S)$ generated by a handle shift $h$ (possibly one-ended) and a mapping class $g$ supported on one of the handles in the support of $h$. (See \cite{PV} for a definition of handle shifts and \cite{REU} for their classification.) Then $\langle g,h \rangle$ is isomorphic to the restricted wreath product $\Z \wr \Z$, which is not virtually abelian and does not contain a nonabelian free group.

Similarly, let $S$ be a surface with a countable set of isolated punctures. We will construct a (one-ended) puncture shift, in analogy with a handle shift. Take a subset $P$ of the countable set of isolated punctures that accumulate to an end $e$. We may index $P$ by $\Z$ as $\{\dots p_{-1},p_0,p_1,\dots\}$. Take a bi-infinite strip $\Sigma$ containing $P$ that has two boundary components that are arcs connecting $e$ to itself, so that $\Sigma$ is homeomorphic to $\R \times [-1,1]$ with punctures at $\Z \times \{0\}$. This is depicted in Figure \ref{fig:punctureshift}. Let $H$ be a homeomorphism that fixes the boundary arcs as well as $S \setminus \Sigma$ and that pushes each $p_i$ to $p_{i+1}$. Let $h$ be the mapping class of $H$ and let $g$ be a nontrivial mapping class supported on the disk containing only $p_0$ and $p_1$. Then the subgroup $\langle g, h^2 \rangle$ is isomorphic to the restricted wreath product $\Z \wr \Z$.

Finally, let $S$ be a surface with a Cantor set of ends and no infinite set of isolated ends. Then none of the ends in the Cantor set $\mathcal{C}$ is accumulated by isolated ends. We will construct a (one-ended) Cantor disk shift. Pick one end $e \in \mathcal{C}$. Partition $\mathcal{C} \setminus e$ into a countable number of Cantor sets and index these by $\Z$: $\{\dots, C_{-1},C_0,C_1,\dots\}$. Let $D_i$ be a disk that supports $C_i$ and no other ends. Let $h$ be a mapping class that takes each $D_i$ to $D_{i+1}$. Let $g$ be any nontrivial mapping class supported in $D_0$. Then the subgroup $\langle g, h \rangle$ is isomorphic to the restricted wreath product $\Z \wr \Z$.
\end{proof}

\begin{figure}
\centering
\includegraphics[scale=1.2]{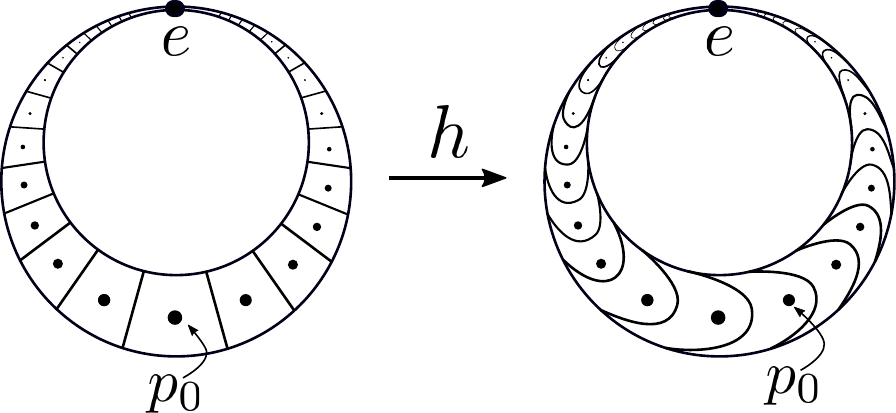}
\caption{A puncture shift.}
\label{fig:punctureshift}
\end{figure}

The classical Tits alternative replaces virtually abelian with virtually solvable. Since the $\Z \wr \Z$ subgroup of $\Mod(S)$ from Proposition \ref{theorem:ZwrZ} is solvable, it is not a counterexample to $\Mod(S)$ satisfying the classical Tits alternative. We now give an an example of a big surface whose mapping class group does not satisfy the classical Tits alternative.

\begin{example}
\label{ex:F}
\textnormal{Thompson's group $F$ does not satisfy the classical Tits alternative, as it is not virtually solvable and does not have a nonabelian free subgroup \cite[Corollary~3.3]{BrinSq}. Let $S$ be \mbox{$\R^2 \setminus \mathcal{C}$}, where $\mathcal{C}$ is a Cantor set embedded in $[0,1] \times \{0\}$. Let $F$ act by piecewise-linear homeomorphisms on $[0,1] \times \R$ so that it also acts faithfully by homeomorphsisms on $\mathcal{C}$. All of these homeomorphisms are orientation preserving and so $F$ is a subgroup of $\Homeo^+(S)$. Since $F$ acts faithfully on $\mathcal C$, the quotient map $\text{Homeo}^+(S) \to \text{Mod}(S)$ is injective on this $F$ subgroup. It follows that $F$ is a subgroup of $\Mod(S)$, which therefore does not satisfy the classical Tits alternative.}
\end{example}

The construction in Example~\ref{ex:F} can be extended to other examples of big surfaces that have a Cantor set in their space of ends. For instance, the construction holds for any surface $S$ where $\E(S)$ contains a Cantor set such that every point in the Cantor set is in the same mapping class group orbit. Still, there are many surfaces where this construction does not hold. 

\begin{question}
\label{q:classical}Do there exist infinite-type surfaces whose mapping class groups satisfy the classical Tits alternative?\end{question}

Another way of framing the Tits alternative is that there are groups that are not subgroups of any finite-type mapping class group. Thompson's group $F$ is of course an example of such a group. On the other hand, considering finite subgroups in the finite-type setting yields a very different phenomenon: every finite group is a subgroup of $\Mod(S_g)$ for some $g$. Two proofs of this fact are given by Farb and Margalit. The first builds a surface modeled on the Cayley graph of the group, and the second uses covering space theory \cite[Theorem~7.12]{Primer}.

When these proofs are extended to the infinite-type setting, they show that every countable group is a subgroup of some big mapping class group. Allcock gave the Cayley graph proof for both finite and infinite-type surfaces \cite{Allcock}. We give here the covering space argument.

\begin{theorem}
\label{thm:everycountable}For every countable group $G$, there is a surface $S$ of infinite type without boundary such that $G < \Mod(S)$. \end{theorem}

\begin{proof}
Let $G$ be a countable group with presentation $\langle \ g_i \ | \ r_i \ \rangle$, $i \in \N$. Let $\Sigma$ be the flute surface $\R^2 \setminus \{(0,0),(1,0),(2,0),\dots\}$. Then $\pi_1(\Sigma) \cong F_\infty$, the (countably) infinite-rank free group. Identify the generators of $G$ with a generating set for $\pi_1(\Sigma)$. Consider the normal subgroup $N$ of $\pi_1(\Sigma)$ equal to the normal closure of all of the relators $r_i$ of $G$. Consider further the cover $S$ of $\Sigma$ corresponding to the normal subgroup $N$. $S$ is a connected infinite-type surface. Note that every sufficiently small neighborhood in $S$ is homeomorphic to some neighborhood in $\Sigma$, and so is either a disk, or a disk with finitely many punctures, or a disk with infinitely many punctures accumulating to a single point.  Further, $G=\langle g_i \ | \ r_i \rangle$ acts by homeomorphisms on $S$ as deck transformations. None of these homeomorphisms are isotopic, since they act differently on (relative) homology classes. Therefore $G<\Mod(S)$.
\end{proof}

Theorem \ref{thm:everycountable}, like Example \ref{ex:F}, proves the existence of big mapping class groups that do not satisfy the classical Tits alternative. Note, however, that Theorem \ref{thm:everycountable} does not by itself resolve Question \ref{q:classical}.

Even if $\Mod(S)$ does not satisfy the classical Tits alternative, it is possible that it satisfies some variant. One result in this direction was proved by Hurtado--Militon, who showed that a version of the Tits alternative holds for the smooth mapping class group of $S$ whenever $S$ has finite genus and a space of ends homeomorphic to a Cantor set \cite[Theorem 1.3]{HM}.

We conclude with a result showing that whenever $\Mod(S)$ does not satisfy the classical Tits alternative, the subgroup given as a counterexample cannot be normal in $\Mod(S)$. Observe that in the case of a surface $S$ with boundary, a normal subgroup consisting of boundary twists is abelian and therefore satisfies the other half of the alternative.

\begin{proposition}

\label{prop:freegroup}
If $S$ is an infinite-type surface without boundary, every nontrivial normal subgroup of $\Mod(S)$ contains a nonabelian free group.
\end{proposition}

\begin{proof}
Let $N$ be any normal subgroup of $\Mod(S)$ and let $f$ be any nontrivial element of $N$. Then by Theorem~\ref{theorem:alexander} there is a curve $c$ such that $f(c) \neq c$. We have that $g=T_{f(c)}T^{-1}_c$ is also in $N$. The element $g$ has support a compact subsurface $\Sigma'$. Let $\Sigma$ be a connected subsurface of $S$ containing $\Sigma'$ such that no boundary component of $\Sigma'$ is homotopic to any boundary component of $\Sigma$. Then $g$ does not act as a Dehn twist about any of the boundary components of $\Sigma$.

Consider $\Mod_{\Sigma}(S)$, the subgroup of $\Mod(S)$ that sends $\Sigma$ to itself. Observe that the restriction $\Mod_{\Sigma}(S)|_{\Sigma}$ is isomorphic to $\Mod(\Sigma)$. By conjugating $g$ by some $h \in \Mod_{\Sigma}(S)$, we may then produce a nonabelian free group as a subgroup of $\langle g, hgh^{-1} \rangle$ by applying the Tits alternative to normal subgroups
in the finite-type setting.\end{proof}

\subsection*{Acknowledgements} The authors would like to thank Jim Belk, Kevin Kordek, and Jing Tao for helpful conversations and Santana Afton, Tyrone Ghaswala, Chris Leininger, Dan Margalit, and Nick Vlamis for their comments on a draft of this note.

\bibliographystyle{plain}
\bibliography{paper}

\end{document}